\documentclass[12pt]{article}
\usepackage{amssymb}
\usepackage{amsthm}
\usepackage{amsmath}
\usepackage{graphicx}
\usepackage{enumerate}

\newtheorem{theorem}{Theorem}[section]
\newtheorem{definition}[theorem]{Definition}
\newtheorem{lemma}[theorem]{Lemma}
\newtheorem{proposition}[theorem]{Proposition}
\newtheorem{remark}[theorem]{Remark}
\newtheorem{example}[theorem]{Example}

\title{Sublattices and $\Delta$-blocks of orthomodular posets}
\author{Ivan~Chajda and Helmut~L\"anger}
\date{}
\begin{document}
\footnotetext[1]{Support of the research by the Austrian Science Fund (FWF), project I~4579-N, and the Czech Science Foundation (GA\v CR), project 20-09869L, entitled ``The many facets of orthomodularity'', as well as by \"OAD, project CZ~02/2019, entitled ``Function algebras and ordered structures related to logic and data fusion'', and, concerning the first author, by IGA, project P\v rF~2020~014, is gratefully acknowledged.}
\maketitle
\begin{abstract}
For orthoposets we introduce a binary relation $\mathrel\Delta$ and a binary operator $d(x,y)$ which are generalizations of the binary relation $\mathrel{\rm C}$ and the commutator $c(x,y)$, respectively, known for orthomodular lattices. We characterize orthomodular posets among orthoposets and orthogonal posets. Moreover, we describe connections between the relations $\mathrel\Delta$ and $\leftrightarrow$ and the operator $d(x,y)$. In details we investigate certain orthomodular posets of subsets of a finite set. In particular we describe maximal orthomodular sublattices and Boolean subalgebras of such orthomodular posets. Finally, we study properties of $\Delta$-blocks with respect to Boolean sublattices and distributive subposets they include.
\end{abstract}



\section{Introduction}

As pointed out firstly in \cite{Bi}, orthomodular lattices and, in particular, orthomodular posets play an important role in the axiomatization of the logic of quantum mechanics. Several books are devoted to orthomodular lattices, cf.\ the monographs \cite{Be} and \cite K, where also some results on orthomodular posets are presented. The monograph \cite{PP} by P.~Pt\'ak and S.~Pulmannov\'a is devoted to the study of $\sigma$-orthocomplete orthomodular posets. Let us mention also several papers on orthomodular posets published by J.~Tkadlec, see e.g.\ \cite{T89} and \cite{T97}. The reader can find a list of sources on this topic in the references of the monograph \cite{PP}.

There are two possible approaches to orthomodular posets:
\begin{itemize}
\item One can study orthomodular posets as partial orthomodular lattices where lattice join is defined for orthogonal elements and then, applying De Morgan's laws, one can derive some lattice meets.
\item One can use the machinery involved for posets used by the authors also in their previous papers \cite{CKoL} -- \cite{CLP18} (partly written with further coauthors) on complemented posets, posets with an antitone involution or weakly orthomodular posets etc.
\end{itemize}
In fact, we will apply here the second approach since it was not used formerly in the quoted sources. Our results are accompanied by examples which will illuminate our concepts and results.

\section{Preliminaries}

Let $(P,\leq)$ be a poset and $'$ a unary operation on $P$. Then $'$ is called an {\em antitone involution} of $(P,\leq)$ if the following conditions hold:
\begin{itemize}
\item if $x\leq y$ then $y'\leq x'$,
\item $x''\approx x$
\end{itemize}
($x,y\in P$). Now let $(P,\leq,0,1)$ be a bounded poset and $'$ a unary operation on $P$. We say that $x,y\in P$ are {\em orthogonal} and write $x\perp y$ if $x\leq y'$. Now $'$ is called a {\em complementation} of $(P,\leq,0,1)$ if
\begin{itemize}
\item $x\vee x'\approx1$ and $x\wedge x'\approx0$
\end{itemize}
($x\in P$). An {\em orthoposet} is an ordered quintuple $\mathbf P=(P,\leq,{}',0,1)$ such that $(P,\leq,0,1)$ is a bounded poset and $'$ is an antitone involution of $(P,\leq)$ and a complementation of $(P,\leq,0,1)$. If, moreover, the following condition holds
\begin{itemize}
\item if $x\perp y$ then $x\vee y$ is defined
\end{itemize}
($x,y\in P$) then $\mathbf P$ is called an {\em orthogonal poset}. If, moreover, $\mathbf P$ satisfies the following condition
\begin{itemize}
\item if $x\leq y$ then $x\vee(x\vee y')'=y$
\end{itemize}
($x,y\in P$) then $\mathbf P$ is called an {\em orthomodular poset}. The last condition is called the {\em orthomodular law}. It is equivalent to its dual form
\begin{itemize}
\item if $x\leq y$ then $(y'\vee(x\vee y')')'=x$
\end{itemize}
($x,y\in P$).

\begin{example}\label{ex2}
The poset $\mathbf P=(P,\leq,{}',0,1)$ depicted in Fig.~1
\vspace*{-2mm}
\begin{center}
\setlength{\unitlength}{7mm}
\begin{picture}(8,8)
\put(4,1){\circle*{.3}}
\put(1,3){\circle*{.3}}
\put(3,3){\circle*{.3}}
\put(5,3){\circle*{.3}}
\put(7,3){\circle*{.3}}
\put(1,5){\circle*{.3}}
\put(3,5){\circle*{.3}}
\put(5,5){\circle*{.3}}
\put(7,5){\circle*{.3}}
\put(4,7){\circle*{.3}}
\put(4,1){\line(-3,2)3}
\put(4,1){\line(-1,2)1}
\put(4,1){\line(1,2)1}
\put(4,1){\line(3,2)3}
\put(4,7){\line(-3,-2)3}
\put(4,7){\line(-1,-2)1}
\put(4,7){\line(1,-2)1}
\put(4,7){\line(3,-2)3}
\put(1,3){\line(0,1)2}
\put(1,3){\line(1,1)2}
\put(3,3){\line(-1,1)2}
\put(3,3){\line(0,1)2}
\put(5,3){\line(0,1)2}
\put(5,3){\line(1,1)2}
\put(7,3){\line(-1,1)2}
\put(7,3){\line(0,1)2}
\put(3.85,.25){$0$}
\put(.3,2.85){$a$}
\put(3.4,2.85){$b$}
\put(4.3,2.85){$c$}
\put(7.4,2.85){$d$}
\put(.3,4.85){$d'$}
\put(3.4,4.85){$c'$}
\put(4.3,4.85){$b'$}
\put(7.4,4.85){$a'$}
\put(3.85,7.4){$1$}
\put(3.2,-.75){{\rm Fig.\ 1}}
\end{picture}
\end{center}
\vspace*{4mm}
is an orthogonal poset, but not orthomodular since $a\leq d'$, but $a\vee(d'\wedge a')=a\vee0=a\neq d'$.
\end{example}

Now let $\mathbf P$ be an orthoposet, $a,b\in P$ and $A,B\subseteq P$. We define
\begin{align*}
L(A) & :=\{x\in P\mid x\leq y\text{ for all }y\in A\}, \\
U(A) & :=\{x\in P\mid y\leq x\text{ for all }y\in A\}.
\end{align*}
Instead of $L(\{a\})$, $L(\{a,b\})$, $L(A\cup\{a\})$, $L(A\cup B)$, $L(U(A))$ we shortly write $L(a)$, $L(a,b)$, $L(A,a)$, $L(A,B)$, $LU(A)$. Analogously we proceed in similar cases. We define binary relations $\mathrel\Delta$ and $\leftrightarrow$ on $P$ as follows:
\begin{align*}
a\mathrel\Delta b & \text{ if and only if }U(a)=U(L(a,b),L(a,b')), \\
 a\leftrightarrow b & \text{ if and only if there exist }c,d,e\in P\text{ with }c\perp d\perp e\perp c, a=c\vee d\text{ and }b=d\vee e
\end{align*}
($a,b\in P$). Since $'$ is an antitone involution on $(P,\leq)$, $a\mathrel\Delta b$ is equivalent to $L(a')=L(U(a',b),U(a',b'))$ ($a,b\in P$). The relation $\leftrightarrow$ for orthomodular posets was introduced in \cite{PP}. Moreover, we put
\[
d(a,b):=U(L(a,b),L(a,b'),L(a',b),L(a',b'))
\]
in arbitrary orthoposets. This operator $d$ generalizes the notion of a commutator for orthomodular lattices (see \cite{Be}) to the case of posets. An {\em ortholattice} (see \cite{Bi}) is an algebra $(L,\vee,\wedge,{}',0,1)$ of type $(2,2,1,0,0)$ such that $(L,\vee,\wedge,0,1)$ is a bounded lattice and $'$ is an antitone involution which is a complementation.  On $L$ we define a binary relation $\mathrel{\rm C}$ (cf.\ \cite{Be}) as follows:
\begin{align*}
a\mathrel{{\rm C}}b & \text{ if and only if }a=(a\wedge b)\vee(a\wedge b')
\end{align*}
($a,b\in L$). An {\em orthomodular lattice} is an ortholattice satisfying the orthomodular law. It can be shown (see \cite K) that in orthomodular lattices $\leftrightarrow$ and $\mathrel{{\rm C}}$ coincide.

\section{Characterization of orthomodular posets}

In this section we study which orthogonal posets are orthomodular. The connection between the the relations $\mathrel{{\rm C}}$ and $\mathrel\Delta$ is as follows.

\begin{lemma}
Let $(L,\vee,\wedge,{}',0,1)$ be an ortholattice and $a,b\in L$. Then $a\mathrel\Delta b$ if and only if $a\mathrel{{\rm C}}b$.
\end{lemma}

\begin{proof}
The following are equivalent:
\begin{align*}
   a & \mathrel\Delta b, \\
U(a) & =U(L(a,b),L(a,b')), \\
U(a) & =U(L(a\wedge b),L(a\wedge b')), \\
U(a) & =U(a\wedge b,a\wedge b'), \\
U(a) & =U((a\wedge b)\vee(a\wedge b')), \\
   a & =(a\wedge b)\vee(a\wedge b'), \\
   a & \mathrel{{\rm C}}b.
\end{align*}
\end{proof}

The following proposition was proved in \cite{PP} for orthomodular posets and in \cite{Be} for orthomodular lattices.

\begin{proposition}\label{prop1}
An orthogonal poset $(P,\leq,{}',0,1)$ is orthomodular if and only if
\[
x\leq y\text{ and }y\wedge x'=0\text{ imply }x=y
\]
{\rm(}$x,y\in P${\rm)}.
\end{proposition}

Hence, an orthogonal poset $(P,\leq,{}',0,1)$ is orthomodular if and only if for every element $x$ of $P$ the element $x'$ is the unique element $y$ of $P$ satisfying $x\perp y$ and $x\vee y=1$.

Let $\mathbf O_6$ denote the orthogonal poset (in fact an ortholattice) depicted in Fig.~2. It is well-known and easy to check that $\mathbf O_6$ is not an orthomodular lattice.
\vspace*{-2mm}
\begin{center}
\setlength{\unitlength}{7mm}
\begin{picture}(4,8)
\put(2,1){\circle*{.3}}
\put(1,3){\circle*{.3}}
\put(3,3){\circle*{.3}}
\put(1,5){\circle*{.3}}
\put(3,5){\circle*{.3}}
\put(2,7){\circle*{.3}}
\put(2,1){\line(-1,2)1}
\put(2,1){\line(1,2)1}
\put(1,3){\line(0,1)2}
\put(3,3){\line(0,1)2}
\put(2,7){\line(-1,-2)1}
\put(2,7){\line(1,-2)1}
\put(1.85,.25){$0$}
\put(.3,2.85){$a$}
\put(3.4,2.85){$b'$}
\put(.3,4.85){$b$}
\put(3.4,4.85){$a'$}
\put(1.85,7.4){$1$}
\put(1.2,-.75){{\rm Fig.\ 2}}
\end{picture}
\end{center}
\vspace*{4mm}

Another characterization of orthomodular posets using $\mathbf O_6$ is as follows.

\begin{proposition}
Let $\mathbf P=(P,\leq,{}',0,1)$ be an orthogonal poset. Then $\mathbf P$ is orthomodular if and only if it does not contain an orthogonal subposet isomorphic to $\mathbf O_6$.
\end{proposition}

\begin{proof}
First assume $\mathbf P$ to be orthomodular. If $\mathbf P$ would contain an orthogonal subposet isomorphic to $\mathbf O_6$ then we would have $a\leq b$, but $a\vee(b\wedge a')=a\vee0=a\neq b$ contradicting orthomodularity. Hence $\mathbf P$ does not contain an orthogonal subposet isomorphic to $\mathbf O_6$. Conversely, assume $\mathbf P$ not to contain an orthogonal subposet isomorphic to $\mathbf O_6$. Suppose, $\mathbf P$ is not orthomodular. Then there exist $c,d\in P$ with $c\leq d$ and $c\vee(d\wedge c')\neq d$. Obviously, $0<c\leq c\vee(d\wedge c')<d<1$. Therefore
\begin{align*}
& 0<c\vee(d\wedge c')<d<1, \\
& 0<d'<c'\wedge(d'\vee c)<1.
\end{align*} If $e\leq d,c'\wedge(d'\vee c)$ then $e\leq d\wedge c',d'\vee c$ and hence $e=0$. This shows
\[
d\wedge(c'\wedge(d'\vee c))=0.
\]
Analogously, one can show
\[
(c\vee(d\wedge c'))\vee c'=1.
\]
Altogether, we see that the subset
\[
\{0,c\vee(d\wedge c'),d,d',c'\wedge(d'\vee c),1\}
\]
of $P$ forms an orthogonal subposet of $\mathbf P$ which is isomorphic to $\mathbf O_6$ contradicting our assumption. Hence $\mathbf P$ is orthomodular.
\end{proof}

One can easily see that the orthogonal poset from Example~\ref{ex2} contains an orthogonal subposet isomorphic to $\mathbf O_6$ (e.g.\ $\{0,b,c,f,g,1\}$) and hence it is not orthomodular.

Using the relation $\mathrel\Delta$ we can easily characterize orthomodular posets among orthogonal posets as follows.

\begin{proposition}
Let $\mathbf P=(P,\leq,{}',0,1)$ be an orthogonal poset. Then $\mathbf P$ is orthomodular if and only if
\[
x\leq y\text{ implies }y\mathrel\Delta x
\]
{\rm(}$x,y\in P${\rm)}.
\end{proposition}

\begin{proof}
Let $a,b\in P$ and assume $a\leq b$. Then the following are equivalent:
\begin{align*}
   b & =a\vee(a'\wedge b), \\
U(b) & =U(a\vee(a'\wedge b)), \\
U(b) & =U(a,a'\wedge b), \\
U(b) & =U(L(a),L(a'\wedge b)), \\
U(b) & =U(L(b,a),L(b,a')), \\
   b & \mathrel\Delta a.
\end{align*}
\end{proof}

\section{Properties of the generalized commutator and \\
commutation relations}

Several important properties of the relation $\leftrightarrow$ are listed in the next lemma.

\begin{lemma}\label{lem2}
{\rm(}cf.\ {\rm\cite{PP})} Let $(P,\leq,{}',0,1)$ be an orthomodular poset and $a,b\in P$. Then the following hold:
\begin{enumerate}[{\rm(i)}]
\item If $a\leq b$ then $a\leftrightarrow b$,
\item if $a\leftrightarrow b$ then $b\leftrightarrow a$, $a\leftrightarrow b'$, $a'\leftrightarrow b$ and $a'\leftrightarrow b'$,
\item if $a\leftrightarrow b$ then $a\vee b,a\wedge b,a\wedge b',a'\wedge b,a'\wedge b'$ are defined and
\begin{align*}
      a & =(a\wedge b)\vee(a\wedge b'), \\
      b & =(a\wedge b)\vee(a'\wedge b), \\
a\vee b & =(a\wedge b)\vee(a\wedge b')\vee(a'\wedge b).
\end{align*}
\end{enumerate}
\end{lemma}

The following properties of the relation $\mathrel\Delta$ and the operator $d$ will be used in the sequel.

\begin{lemma}\label{lem3}
Let $(P,\leq,{}',0,1)$ be an orthoposet and $a,b\in P$. Then the following hold:
\begin{enumerate}[{\rm(i)}]
\item If $a\leq b$ then $a\mathrel\Delta b$,
\item $a\mathrel\Delta0$, $a\mathrel\Delta a'$ and $1\mathrel\Delta a$,
\item $a\mathrel\Delta b$ if and only if $a\mathrel\Delta b'$,
\item $d(a,b)=d(b,a)=d(a,b')=d(a',b)=d(a',b')$,
\item if $a\mathrel\Delta b$ and $a'\mathrel\Delta b$ then $d(a,b)=\{1\}$,
\item $d(a,0)=d(a,a)=d(a,a')=d(a,1)=\{1\}$.
\end{enumerate}
\end{lemma}

\begin{proof}
\
\begin{enumerate}
\item[(i)] If $a\leq b$ then
\[
U(a)=UL(a)=U(L(a),L(a,b'))=U(L(a,b),L(a,b')).
\]
\item[(ii)] We have
\begin{align*}
U(a) & =U(0,a)=U(0,L(a))=U(L(a,0),L(a,0')), \\
U(a) & =U(0,a)=U(0,L(a))=U(L(a,a'),L(a,a''), \\
U(1) & =U(a,a')=U(L(a),L(a'))=U(L(1,a),L(1,a')).
\end{align*}
\item[(iii)] and (iv) are clear.
\item[(v)] If $a\mathrel\Delta b$ and $a'\mathrel\Delta b$ then
\begin{align*}
d(a,b) & =U(L(a,b),L(a,b'),L(a',b),L(a',b'))= \\
       & =U(L(a,b),L(a,b'))\cap U(L(a',b),L(a',b'))=U(a)\cap U(a')=U(a,a')= \\
			 & =\{1\}.
\end{align*}
\item[(vi)] This follows from (i), (ii) and (v).
\end{enumerate}
\end{proof}

In the next lemma we show connections among the relations $\mathrel\Delta$, $\leftrightarrow$ and the operator $d$.

\begin{lemma}\label{lem4}
Let $(P,\leq,{}',0,1)$ be an orthomodular poset and $a,b\in P$. Then the following hold:
\begin{enumerate}[{\rm(i)}]
\item if $a\leftrightarrow b$ then $a\mathrel\Delta b$, $b\mathrel\Delta a$ and $d(a,b)=\{1\}$,
\item if $a\leq b$ then $a\mathrel\Delta b$, $b\mathrel\Delta a$ and $d(a,b)=\{1\}$.
\end{enumerate}
\end{lemma}

\begin{proof}
\
\begin{enumerate}[(i)]
\item Assume $a\leftrightarrow b$. Then
\[
U(a)=U((a\wedge b)\vee(a\wedge b'))=U(a\wedge b,a\wedge b')=U(L(a,b),L(a,b'))
\]
and hence $a\mathrel\Delta b$. The rest follows from (ii) of Lemma~\ref{lem2} and (v) of Lemma~\ref{lem3}.
\item This follows from (i) of Lemma~\ref{lem2} and from (i).
\end{enumerate}
\end{proof}

In Theorem~\ref{th1} we show that the converse of (ii) does not hold in general.

\begin{lemma}\label{lem1}
Let $(P,\leq,{}',0,1)$ be an orthomodular poset and $a,b\in P$ and assume $a\mathrel{{\rm C}}b$ and that $a\vee b$ is defined. Then
\begin{enumerate}[{\rm(i)}]
\item $a\vee b=(a\wedge b')\vee b$,
\item if $(b\wedge a)\vee(b\wedge a')$ and $a\vee b'$ are defined then $b\mathrel{{\rm C}}a$ and $a\vee b=a\vee(b\wedge a')$.
\end{enumerate}
\end{lemma}

\begin{proof}
\
\begin{enumerate}[(i)]
\item Because of $a\mathrel{{\rm C}}b$ we have $a=(a\wedge b)\vee(a\wedge b')$ and hence $a\vee b=(a\wedge b)\vee(a\wedge b')\vee b=(a\wedge b')\vee b$.
\item Assume that $(b\wedge a)\vee(b\wedge a')$ and $a\vee b'$ are defined. Now $a\mathrel{{\rm C}}b$ implies $a\mathrel{{\rm C}}b'$. Since $a\vee b'$ is defined we have $a\vee b'=(a\wedge b)\vee b'$ by (i). Now $(b\wedge a)\vee(b\wedge a')\leq b$. According to orthomodularity we obtain
\begin{align*}
b & =(b\wedge a)\vee(b\wedge a')\vee((b\wedge a)\vee(b\wedge a')\vee b')'= \\
  & =(b\wedge a)\vee(b\wedge a')\vee((a\vee b')\vee(b\wedge a'))'=(b\wedge a)\vee(b\wedge a')\vee1'= \\
  & =(b\wedge a)\vee(b\wedge a')\vee0=(b\wedge a)\vee(b\wedge a'),
\end{align*}
i.e.\ $b\mathrel{{\rm C}}a$. Hence
\[
a\vee b=a\vee(b\wedge a)\vee(b\wedge a')=a\vee(b\wedge a').
\]
\end{enumerate}
\end{proof}

If $a\mathrel\Delta b$ is assumed instead of $a\mathrel{{\rm C}}b$, we can modify (i) of Lemma~\ref{lem1} as follows.

\begin{lemma}
Let $(P,\leq,{}',0,1)$ be an orthomodular poset and $a,b\in P$ and assume $a\mathrel{\Delta}b$. Then $U(a,b)=U(L(a,b'),b)$.
\end{lemma}

\begin{proof}
We have
\begin{align*}
U(a,b) & =U(a)\cap U(b)=U(L(a,b),L(a,b'))\cap U(b)=UL(a,b)\cap UL(a,b')\cap U(b)= \\
       & =UL(a,b')\cap U(b)=U(L(a,b'),b).
\end{align*}
\end{proof}

\section{Sublattices of orthomodular posets of subsets of a finite set}

We now introduce a particular orthomodular poset whose elements are special subsets of a given $n$-element set as follows:

\begin{definition}
Let $n,k$ be positive integers with $k\mid n$ and put
\begin{align*}
             N & :=\{1,\ldots,n\}, \\
        P_{nk} & :=\{A\subseteq N\mid k\text{ divides }|A|\}, \\
            A' & :=N\setminus A\text{ for all }A\in P_{nk}, \\
\mathbf P_{nk} &:=(P_{nk},\subseteq,{}',\emptyset,N)
\end{align*}
\end{definition}

It is evident that $\mathbf P_{n1}$ is just the Boolean algebra $\mathbf2^n$ of all subsets of the $n$-element set $N$. It is also clear that every orthomodular poset $\mathbf P_{nk}$ can be embedded as a poset into the Boolean algebra $\mathbf2^n$. Altogether, it seems that it is more suitable to investigate these orthomodular posets than the general case.

It should be noted that the orthomodular poset $\mathbf P_{n2}$ was already introduced in \cite{PP} under the name $\mathbf P_{{\rm even}}$.

It is easy to see that if $\{A_1,\ldots,A_{n/k}\}$ is a decomposition on $N$ into $k$-element subsets then
\[
\{\bigcup_{i\in I}A_i\mid I\subseteq\{1,\ldots,n/k\}\}
\]
forms a maximal Boolean subalgebra of $\mathbf P_{nk}$ isomorphic to $\mathbf2^{n/k}$.

We are now going to describe several important properties of the orthomodular poset $\mathbf P_{nk}$.

\begin{theorem}\label{th1}
Let $n,k$ be positive integers with $k\mid n$ and $A,B,C\in P_{nk}$ and put $N:=\{1,\ldots,n\}$. Then
\begin{enumerate}[{\rm(i)}]
\item $\mathbf P_{nk}$ is an orthomodular poset,
\item $A\perp B$ if and only if $A\cap B=\emptyset$,
\item $A\leftrightarrow B$ if and only if $A\cap B\in P_{nk}$,
\item $\mathbf P_{nk}$ is an orthomodular lattice if and only if $k=1$ {\rm(}then it is a Boolean algebra{\rm)} or $n/k\leq2$,
\item if $|A\cap B|<k$ then $UL(A,B)=P_{nk}$,
\item if $|A\cap B|\geq k$ then $C\in UL(A,B)$ if and only if $C\supseteq A\cap B$,
\item if $|A\cap B|,|A\cap B'|\geq k$ then $A\mathrel\Delta B$,
\item if $|A\cap B|,|A\cap B'|,|A'\cap B|,|A'\cap B'|\geq k$ then $d(A,B)=\{N\}$,
\item if $n\geq4k$ and $k>1$ then there exist $D,E\in P_{nk}$ with $D\mathrel\Delta E$ and $E\not\mathrel\Delta D$,
\item if $n\geq6k$ and $k>1$ then there exist $D,E\in P_{nk}$ with $D\not\leftrightarrow E$ and $d(D,E)=\{N\}$,
\item if $|A|=k$ then $A\mathrel\Delta B$ if and only if $A\leftrightarrow B$.
\end{enumerate}
\end{theorem}

\begin{proof}
\
\begin{enumerate}
\item[(ii)] The following are equivalent: $A\perp B$, $A\subseteq B'$ and $A\cap B=\emptyset$.
\item[(i)] Obviously, $\mathbf P_{nk}$ is an orthoposet. If $A\perp B$ then $A\cap B=\emptyset$ as shown above and hence $A\cup B\in P_{nk}$ which is clearly the supremum $A\vee B$ of $A$ and $B$ in $\mathbf P_{nk}$. This shows that $\mathbf P_{nk}$ is an orthogonal poset. Finally, $A\subseteq B$ implies
\[
A\vee(A\vee B')'=A\cup(A\cup B')'=A\cup(A'\cap B)=(A\cup A')\cap(A\cup B)=N\cap B=B
\]
showing that $\mathbf P_{nk}$ is an orthomodular poset.
\item[(iii)] Assume $A\leftrightarrow B$. Then there exist pairwise disjoint $C,D,E\in P_{nk}$ with $C\cup D=A$ and $D\cup E=B$. Obviously, $D\subseteq A\cap B$. Conversely, assume $a\in A\cap B$. Then $a\notin D$ would imply $a\in C\cap E$ contradicting $C\cap E=\emptyset$. Hence $a\in D$. This shows $D=A\cap B$ and hence $A\cap B=D\in P_{nk}$. The converse implication is clear.
\item[(iv)] We have $P_{n1}=2^N$ and $P_{nn}=\{\emptyset,N\}$, if $n/k=2$ then $P_{nk}=\{0,N\}\cup\{A\subseteq N\mid k=|A|\}$ and if $k>1$ and $n/k>2$ then $\{1,\ldots,2k\}$ and $\{2,\ldots,2k+1\}$ are different minimal upper bounds of $\{2,\ldots,k+1\}$ and $\{3,\ldots,k+2\}$.
\item[(v)] If $|A\cap B|<k$ then $UL(A,B)=U(\{\emptyset\})=P_{nk}$.
\item[(vi)] If $|A\cap B|\geq k$, $C\in UL(A,B)$ and $a\in A\cap B$ then there exists some $D\in L(A,B)$ with $a\in D$ and since $C\supseteq D$, we obtain $a\in C$ showing $A\cap B\subseteq C$.
\item[(vii)] if $|A\cap B|,|A\cap B'|\geq k$ then by (vi) the following are equivalent:
\begin{align*}
C & \in U(A), \\
C & \supseteq A, \\
C & \supseteq(A\cap B)\cup(A\cap B'), \\
C & \supseteq A\cap B\text{ and }C\supseteq A\cap B', \\
C & \in UL(A,B)\cap UL(A,B'), \\
C & \in U(L(A,B),L(A,B'))
\end{align*}
and hence $U(A)=U(L(A,B),L(A,B'))$, i.e.\ $A\mathrel\Delta B$.
\item[(viii)] This follows from (vii) and from (iv) of Lemma~\ref{lem3}.
\item[(ix)] Assume $n\geq4k$ and $k>1$ and put $D:=\{1,\ldots,3d\}$ and $E:=\{2d,\ldots,4d-1\}$. Then
\begin{align*}
 |D\cap E| & =k+1>k, \\
|D\cap E'| & =2k-1>k, \\
|D'\cap E| & =k-1<k
\end{align*}
and hence $D\mathrel\Delta E$ by (iv). Put $F:=\{k+1,\ldots,3k\}$. Since $F\supseteq E\cap D$ we have $F\in UL(E,D)$ and since $|E\cap D'|<k$ we have $UL(E,D')=P_{nk}$ by (v). Together we obtain $F\in UL(E,D)\cap UL(E,D')=U(L(E,D),L(E,D'))$. But, because of $k>1$ we have $3k<4k-1$ and hence $4k-1\in E\setminus F$ which shows $F\notin U(E)$. Together we obtain $U(E)\neq U(L(E,D),L(E,D'))$, i.e.\ $E\not\mathrel\Delta D$.
\item[(x)] If $n\geq6k$, $D:=\{1,\ldots,3d\}$ and $E:=\{2d,\ldots,5d-1\}$ then
\begin{align*}
  |D\cap E| & =k+1>k, \\
 |D\cap E'| & =|D'\cap E|=2k-1>k, \\
|D'\cap E'| & =n-(k+1)-(4k-2)=n-5k+1>k
\end{align*}
and hence $D\not\leftrightarrow E$ by (iii) and $d(D,E)=\{N\}$ by (viii).
\item[(xi)] Assume $A\not\leftrightarrow B$. Then $A\not\subseteq B$ and $A\not\subseteq B'$ by Lemma~\ref{lem2}. Hence $0<|A\cap B|,|A\cap B'|<d$ and $B\notin U(A)$. But
\[
B\in P_{nk}\cap P_{nk}=UL(A,B)\cap UL(A,B')=U(L(A,B),L(A,B'))
\]
by (v) showing $U(A)\neq U(L(A,B),L(A,B'))$, i.e.\ $A\not\mathrel\Delta B$. The converse implication follows from Lemma~\ref{lem4}.
\end{enumerate}
\end{proof}

\begin{example}\label{ex1}
Put $N:=\{1,\ldots,6\}$ and
\[
P:=\{A\subseteq N\mid|A\cap\{1,2,3\}|=|A\cap\{4,5,6\}|\}.
\]
Then $\mathbf P=(P,\subseteq,{}',\emptyset,N)$ is a twenty-element orthomodular poset which is not a lattice since
\[
\{1,4\},\{1,5\},\{1,2,4,5\},\{1,3,4,5\}\in P
\]
and $\{1,2,4,5\}$ and $\{1,3,4,5\}$ are different minimal upper bounds of $\{1,4\}$ and $\{1,5\}$. It is the smallest orthomodular subposet of the orthomodular poset $\mathbf P_{62}$ {\rm(}see Proposition~\ref{prop2}{\rm)} and it is depicted in Fig.~3:
\vspace*{-2mm}
\begin{center}
\setlength{\unitlength}{7mm}
\begin{picture}(18,8)
\put(9,1){\circle*{.3}}
\put(1,3){\circle*{.3}}
\put(3,3){\circle*{.3}}
\put(5,3){\circle*{.3}}
\put(7,3){\circle*{.3}}
\put(9,3){\circle*{.3}}
\put(11,3){\circle*{.3}}
\put(13,3){\circle*{.3}}
\put(15,3){\circle*{.3}}
\put(17,3){\circle*{.3}}
\put(1,5){\circle*{.3}}
\put(3,5){\circle*{.3}}
\put(5,5){\circle*{.3}}
\put(7,5){\circle*{.3}}
\put(9,5){\circle*{.3}}
\put(11,5){\circle*{.3}}
\put(13,5){\circle*{.3}}
\put(15,5){\circle*{.3}}
\put(17,5){\circle*{.3}}
\put(9,7){\circle*{.3}}
\put(1,3){\line(0,2)2}
\put(1,3){\line(1,1)2}
\put(1,3){\line(3,1)6}
\put(1,3){\line(4,1)8}
\put(3,3){\line(-1,1)2}
\put(3,3){\line(1,1)2}
\put(3,3){\line(2,1)4}
\put(3,3){\line(4,1)8}
\put(5,3){\line(-1,1)2}
\put(5,3){\line(0,1)2}
\put(5,3){\line(2,1)4}
\put(5,3){\line(3,1)6}
\put(7,3){\line(-3,1)6}
\put(7,3){\line(-2,1)4}
\put(7,3){\line(3,1)6}
\put(7,3){\line(4,1)8}
\put(9,3){\line(-4,1)8}
\put(9,3){\line(-2,1)4}
\put(9,3){\line(2,1)4}
\put(9,3){\line(4,1)8}
\put(11,3){\line(-4,1)8}
\put(11,3){\line(-3,1)6}
\put(11,3){\line(2,1)4}
\put(11,3){\line(3,1)6}
\put(13,3){\line(-3,1)6}
\put(13,3){\line(-2,1)4}
\put(13,3){\line(0,1)2}
\put(13,3){\line(1,1)2}
\put(15,3){\line(-4,1)8}
\put(15,3){\line(-2,1)4}
\put(15,3){\line(-1,1)2}
\put(15,3){\line(1,1)2}
\put(17,3){\line(-4,1)8}
\put(17,3){\line(-3,1)6}
\put(17,3){\line(-1,1)2}
\put(17,3){\line(0,1)2}
\put(9,1){\line(-4,1)8}
\put(9,1){\line(-3,1)6}
\put(9,1){\line(-2,1)4}
\put(9,1){\line(-1,1)2}
\put(9,1){\line(0,1)2}
\put(9,1){\line(1,1)2}
\put(9,1){\line(2,1)4}
\put(9,1){\line(3,1)6}
\put(9,1){\line(4,1)8}
\put(9,7){\line(-4,-1)8}
\put(9,7){\line(-3,-1)6}
\put(9,7){\line(-2,-1)4}
\put(9,7){\line(-1,-1)2}
\put(9,7){\line(0,-1)2}
\put(9,7){\line(1,-1)2}
\put(9,7){\line(2,-1)4}
\put(9,7){\line(3,-1)6}
\put(9,7){\line(4,-1)8}
\put(8.85,.25){$\emptyset$}
\put(.85,2.3){$a$}
\put(2.85,2.3){$b$}
\put(4.85,2.3){$c$}
\put(6.85,2.3){$d$}
\put(8.85,2.3){$e$}
\put(10.85,2.3){$f$}
\put(12.85,2.3){$g$}
\put(14.85,2.3){$h$}
\put(16.85,2.3){$i$}
\put(.8,5.4){$i'$}
\put(2.8,5.4){$h'$}
\put(4.8,5.4){$g'$}
\put(6.8,5.4){$f'$}
\put(8.8,5.4){$e'$}
\put(10.8,5.4){$d'$}
\put(12.8,5.4){$c'$}
\put(14.8,5.4){$b'$}
\put(16.8,5.4){$a'$}
\put(8.75,7.4){$N$}
\put(8.2,-.75){{\rm Fig.\ 3}}
\end{picture}
\end{center}
\vspace*{4mm}
with $a=\{1,4\}$, $b=\{1,5\}$, $c=\{1,6\}$, $d=\{2,4\}$, $e=\{2,5\}$, $f=\{2,6\}$, $g=\{3,4\}$, $h=\{3,5\}$, $i=\{3,6\}$, $a'=\{2,3,5,6\}$, $b'=\{2,3,4,6\}$, $c'=\{2,3,4,5\}$, $d'=\{1,3,5,6\}$, $e'=\{1,3,4,6\}$, $f'=\{1,3,4,5\}$, $g'=\{1,2,5,6\}$, $h'=\{1,2,4,6\}$ and $i'=\{1,2,4,5\}$. If 
\begin{align*}
& B_1:=\{0,a,e,i,a',e',i',N\}, \\
& B_2:=\{0,a,f,h,a',f',h',N\}, \\
& B_3:=\{0,b,d,i,b',d',i',N\}, \\
& B_4:=\{0,b,f,g,b',f',g',N\}, \\
& B_5:=\{0,c,d,h,c',d',h',N\}, \\
& B_6:=\{0,c,e,g,c',e',g',N\}
\end{align*}
then $\mathbf B_i:=(B_i,\cup,\cap,{}',\emptyset,N)$ is a maximal Boolean subalgebra of $\mathbf P$ for all $i\in N$. The Hasse diagram of $\mathbf B_2$ is depicted in Fig.~4:
\vspace*{-2mm}
\begin{center}
\setlength{\unitlength}{7mm}
\begin{picture}(6,8)
\put(3,1){\circle*{.3}}
\put(1,3){\circle*{.3}}
\put(3,3){\circle*{.3}}
\put(5,3){\circle*{.3}}
\put(1,5){\circle*{.3}}
\put(3,5){\circle*{.3}}
\put(5,5){\circle*{.3}}
\put(3,7){\circle*{.3}}
\put(3,1){\line(-1,1)2}
\put(3,1){\line(0,1)2}
\put(3,1){\line(1,1)2}
\put(3,7){\line(-1,-1)2}
\put(3,7){\line(0,-1)2}
\put(3,7){\line(1,-1)2}
\put(3,3){\line(-1,1)2}
\put(3,3){\line(1,1)2}
\put(3,5){\line(-1,-1)2}
\put(3,5){\line(1,-1)2}
\put(1,3){\line(2,1)4}
\put(5,3){\line(-2,1)4}
\put(2.85,.25){$\emptyset$}
\put(.3,2.85){$a$}
\put(3.4,2.85){$f$}
\put(5.4,2.85){$h$}
\put(.3,4.85){$a'$}
\put(3.4,4.85){$f'$}
\put(5.4,4.85){$h'$}
\put(2.75,7.4){$N$}
\put(2.2,-.75){{\rm Fig.\ 4}}
\end{picture}
\end{center}
\vspace*{4mm}
\end{example}

\begin{proposition}\label{prop2}
The orthomodular poset $\mathbf P=(P,\subseteq,{}',\emptyset,N)$ from Example~\ref{ex1} with $N:=\{1,\ldots,6\}$ and $C':=N\setminus C$ for all $C\in P$ is {\rm(}up to isomorphism{\rm)} the smallest orthomodular subposet of $\mathbf P_{62}$ that is not a lattice.
\end{proposition}

\begin{proof}
Let $\mathbf Q=(Q,\leq,{}',0,1)$ be an orthomodular subposet of $\mathbf P_{62}$ that is not a lattice. Then there exist $A,B\in Q$ such that $A\vee B$ is not defined. From this we conclude that there exist five pairwise different elements $a,b,c,d,e$ of $N$ such that $A=\{a,b\}$ and $B=\{a,c\}$ and such that $\{a,b,c,d\},\{a,b,c,e\}\in Q$. Without loss of generality assume $\{1,4\},\{1,5\},\{1,2,4,5\},\{1,3,4,5\}\in Q$. Now $C':=N\setminus C\in Q$ for all $C\in Q$ and $C\cup D\in Q$ for all $C,D\in Q$ with $C\cap D=\emptyset$. Because of these rules one obtains in finitely many steps $P\subseteq Q$.
\end{proof}

There arises the natural question to determine maximal orthomodular sublattices of a given orthomodular poset. For the orthomodular poset from Example~\ref{ex1} we solve this problem as follows.

\begin{proposition}
Consider the orthomodular poset $\mathbf P$ from Example~\ref{ex1}, put $Q:=\{A\in P\mid 2=|A|\}$ and for every $A\in Q$ put $P_A:=\{B\in P\mid B\supseteq A\text{ or } B\subseteq A'\}$. Moreover, define binary operations $\vee$ and $\wedge$ on $P_A$ as follows:
\begin{align*}
  B\vee C & :=\left\{
\begin{array}{ll}
B\cup C        & \text{if }|B\cup C|\text{ is even}, \\
B\cup C\cup A' & \text{otherwise}
\end{array}
\right. \\
B\wedge C & :=\left\{
\begin{array}{ll}
B\cap C       & \text{if }|B\cap C|\text{ is even}, \\
B\cap C\cap A & \text{otherwise}
\end{array}
\right.
\end{align*}
{\rm(}$B,C\in P_A${\rm)}. Then $\mathbf P_A:=(P_A,\vee,\wedge,{}',\emptyset,N),A\in Q,$ are nine pairwise distinct twelve-element maximal orthomodular sublattices of $\mathbf P$ each of which is the atomic pasting {\rm(}see e.g.\ {\rm\cite{Be})} of two eight-element Boolean algebras via the atom $A$.
\end{proposition}

\begin{proof}
Let $A,B\in Q$. It is easy to see that $(P_A,\vee,\wedge,{}',\emptyset,N)$ is an orthomodular sublattice of $\mathbf P$. Assume that $\mathbf P_A$ is not a maximal orthomodular sublattice of $\mathbf P$. Then there exists an orthomodular sublattice $\mathbf R=(R,\vee,\wedge,{}',\emptyset,N)$ of $\mathbf P$ with $R\supsetneqq P_A$. Let $C\in R\setminus P_A$. Then $C\cap A,C'\cap A\neq\emptyset$. Let $D$ denote the four-element member of the set $\{C,C'\}$. Since $|D\cap A'|=3$ there exist two different two-element subsets of $D\cap A'$. Now $A'$ and $D$ are different minimal upper bounds of $E$ and $F$ which shows that $E\vee F$ does not exist in $\mathbf R$ contradicting the fact that $\mathbf R$ is a lattice. Hence $\mathbf P_A$ is a maximal orthomodular sublattice of $\mathbf P$. Now assume $A\neq B$. If $A\cap B\neq\emptyset$ then $A\in P_A\setminus P_B$ and hence $P_A\neq P_B$. If $A\cap B=\emptyset$ then
\begin{align*}
\{C\in P_A\mid C\supseteq A\} & =\{A\}\cup\{C\in P\mid C\supseteq A\text{ and }|C|=4\}\cup\{N\}, \\
\{C\in P_B\mid C\supseteq A\} & =\{A,B',A\cup B,N\}
\end{align*}
and hence
\[
|\{C\in P_A\mid C\supseteq A\}|=1+4+1=6\neq4=|\{C\in P_B\mid C\supseteq A\}|
\]
which yields $P_A\neq P_B$. This shows that the sets $P_G,G\in Q,$ are pairwise distinct.
\end{proof}

\begin{example}\label{ex3}
For the orthomodular poset $\mathbf P$ from Example~\ref{ex1} we obtain the following maximal orthomodular sublattices $\mathbf P_x$ with base set $P_x$:
\begin{align*}
P_a & =B_1\cup B_2, \\
P_b & =B_3\cup B_4, \\
P_c & =B_5\cup B_6, \\
P_d & =B_3\cup B_5, \\
P_e & =B_1\cup B_6, \\
P_f & =B_2\cup B_4, \\
P_g & =B_4\cup B_6, \\
P_h & =B_2\cup B_5, \\
P_i & =B_1\cup B_3.
\end{align*}
Clearly, every of these orthomodular lattices $\mathbf P_x$ is the atomic pasting {\rm(}see e.g.\ {\rm\cite{Be})} of two eight-element Boolean algebras via the atom $x$ of $\mathbf P_x$. The Hasse diagram of $\mathbf P_a$ is depicted in Fig.~5:
\vspace*{-2mm}
\begin{center}
\setlength{\unitlength}{7mm}
\begin{picture}(10,8)
\put(5,1){\circle*{.3}}
\put(1,3){\circle*{.3}}
\put(3,3){\circle*{.3}}
\put(5,3){\circle*{.3}}
\put(7,3){\circle*{.3}}
\put(9,3){\circle*{.3}}
\put(1,5){\circle*{.3}}
\put(3,5){\circle*{.3}}
\put(5,5){\circle*{.3}}
\put(7,5){\circle*{.3}}
\put(9,5){\circle*{.3}}
\put(5,7){\circle*{.3}}
\put(5,1){\line(-2,1)4}
\put(5,1){\line(-1,1)4}
\put(5,1){\line(0,1)2}
\put(5,1){\line(1,1)4}
\put(5,1){\line(2,1)4}
\put(5,7){\line(-2,-1)4}
\put(5,7){\line(-1,-1)4}
\put(5,7){\line(0,-1)2}
\put(5,7){\line(1,-1)4}
\put(5,7){\line(2,-1)4}
\put(5,3){\line(-2,1)4}
\put(5,3){\line(-1,1)2}
\put(5,3){\line(1,1)2}
\put(5,3){\line(2,1)4}
\put(5,5){\line(-2,-1)4}
\put(5,5){\line(-1,-1)2}
\put(5,5){\line(1,-1)2}
\put(5,5){\line(2,-1)4}
\put(4.85,.25){$\emptyset$}
\put(.3,2.85){$e$}
\put(2.3,2.85){$i$}
\put(5.4,2.85){$a$}
\put(7.4,2.85){$f$}
\put(9.4,2.85){$h$}
\put(.3,4.85){$e'$}
\put(2.3,4.85){$i'$}
\put(5.4,4.85){$a'$}
\put(7.4,4.85){$f'$}
\put(9.4,4.85){$h'$}
\put(4.75,7.4){$N$}
\put(4.2,-.75){{\rm Fig.\ 5}}
\end{picture}
\end{center}
\vspace*{4mm}
\end{example}

We can derive the following more general result.

\begin{theorem}\label{th2}
Let $k$ be an integer $>1$, put $n:=3k$, $N:=\{1,\ldots,n\}$ and $B':=N\setminus B$ for all $B\in2^N$, let $Q$ denote the set of $k$-element subsets of $N$, put $P_B:=\{C\in P_{nk}\mid C\supseteq B\text{ or }C\subseteq B'\}$ and $\mathbf P_B:=(P_B,\subseteq,{}',\emptyset,N)$ for all $B\in Q$ and let $A\in Q$. Then
\begin{enumerate}[{\rm(i)}]
\item $\mathbf P_A$ is a maximal orthomodular sublattice of the orthomodular poset $\mathbf P_{nk}$,
\item $|P_A|=4+2\binom{2k}k$,
\item $\mathbf P_A$ is the atomic pasting of $\binom{2k}k/2$ eight-element Boolean algebras via the atom $A$,
\item $\mathbf P_B,B\in Q,$ are pairwise distinct.
\end{enumerate}
\end{theorem}

\begin{proof}
Let $B,C\in P_A$.
\begin{enumerate}
\item[(ii)]
We have
\begin{align*}
 \{D\in P_A\mid|D|=0\} & =\{\emptyset\}, \\
 \{D\in P_A\mid|D|=k\} & =\{A\}\cup\{D\subseteq A'\mid|D|=k\}, \\
\{D\in P_A\mid|D|=2k\} & =\{A'\}\cup\{D\supseteq A\mid|D|=2k\}, \\
\{D\in P_A\mid|D|=3k\} & =\{N\}.
\end{align*}
and hence
\[
|P_A|=1+1+\binom{2k}k+1+\binom{2k}k+1=4+\binom{2k}k.
\]
\item[(i)]
First assume $k\nmid|B\cup C|$.
\begin{align*}
                       \text{If }A\subseteq B,C & \text{ then }|B|=|C|=2k\text{ and }|B\cup C|>2k\text{ and hence} \\
										                          	& B\vee C=N=B\cup C\cup A', \\
\text{if }A\subseteq B\text{ and }C\subseteq A' & \text{ then }|B|=2k,|C|=k\text{ and }|B\cup C|>2k\text{ and hence} \\
                                                & B\vee C=N=B\cup C\cup A', \\
\text{if }B\subseteq A'\text{ and }A\subseteq C & \text{ then }|B|=k,|C|=2k\text{ and }|B\cup C|>2k\text{ and hence} \\
                                                & B\vee C=N=B\cup C\cup A', \\
                      \text{if }B,C\subseteq A' & \text{ then }|B|=|C|=k\text{ and }k<|B\cup C|<2k\text{ and hence} \\
											                          & B\vee C=A'=B\cup C\cup A'.
\end{align*}
Now assume $k\nmid|B\cap C|$.
\begin{align*}
                       \text{If }A\subseteq B,C & \text{ then }|B|=|C|=2k\text{ and }k<|B\cap C|<2k\text{ and hence} \\
                                                & B\wedge C=A=B\cap C\cap A, \\
\text{if }A\subseteq B\text{ and }C\subseteq A' & \text{ then }|B|=2k,|C|=k\text{ and }|B\cap C|<k\text{ and hence} \\
                                                & B\wedge C=\emptyset=B\cap C\cap A, \\
\text{if }B\subseteq A'\text{ and }A\subseteq C & \text{ then }|B|=k,|C|=2k\text{ and }|B\cap C|<k\text{ and hence} \\
                                                & B\wedge C=\emptyset=B\cap C\cap A, \\
                      \text{if }B,C\subseteq A' & \text{ then }|B|=|C|=k\text{ and }|B\cap C|<k\text{ and hence } \\
                                                & B\wedge C=\emptyset=B\cap C\cap A.
\end{align*}
Now let $B,C$ be arbitrary elements of $P_A$. Thus we have just proved:
\begin{align*}
  B\vee C & :=\left\{
\begin{array}{ll}
B\cup C        & \text{if }|B\cup C|\text{ is even}, \\
B\cup C\cup A' & \text{otherwise}
\end{array}
\right. \\
B\wedge C & :=\left\{
\begin{array}{ll}
B\cap C       & \text{if }|B\cap C|\text{ is even}, \\
B\cap C\cap A & \text{otherwise}
\end{array}
\right.
\end{align*}
Hence $\mathbf P_A$ is an orthomodular sublattice of the orthomodular poset $\mathbf P_{nk}$. Assume that $\mathbf P_A$ is not a maximal orthomodular sublattice of the orthomodular poset $\mathbf P_{nk}$. Then there exists an orthomodular sublattice $\mathbf R=(R,\vee,\wedge,{}',\emptyset,N)$ of the orthomodular poset $\mathbf P_{nk}$ with $R\supsetneqq P_A$. Let $D\in R\setminus P_A$. Then $D\cap A,D'\cap A\neq\emptyset$. Let $E$ denote the $2k$-element member of the set $\{D,D'\}$. Since $|E\cap A|<k$ we have $|E\cap A'|>k$. Let $F,G$ be two different $k$-element subsets of $E\cap A'$. Then $A'$ and $E$ are different minimal upper bounds of $F$ and $G$ which shows that $F\vee G$ does not exist in $\mathbf R$ contradicting the fact that $\mathbf R$ is a lattice. Hence $\mathbf P_A$ is a maximal orthomodular sublattice of the orthomodular poset $\mathbf P_{nk}$.
\item[(iii)] Let $a\in A'$ and put
\[
S:=\{D\subseteq A'\mid|D|=k\text{ and }a\in D\}
\]
and
\[
B_D:=\{\emptyset,A,D,(A\cup D)',A\cup D,D',A',N\}
\]
and $\mathbf B_D:=(B_D,\cup,\cap,{}',\emptyset,N)$ for all $D\in S$. Then
\[
|S|=\frac12\binom{2k}k
\]
since $D\mapsto A'\setminus D$ is a bijection between $S$ and $\{D\subseteq A'\mid|D|=k\text{ and }a\notin D\}$. It is easy to see that $\mathbf B_D,D\in S,$ are
\[
\frac12\binom{2k}k
\]
pairwise different eight-element Boolean subalgebras of $\mathbf P_{nk}$ and
\[
\bigcup_{D\in S}B_D=P_A.
\]
Since
\[
\frac12\binom{2k}k4+4=4+2\binom{2k}k=|P_A|,
\]
these Boolean subalgebras have only $\emptyset,A,A',N$ in common. This means that $\mathbf P_A$ is the atomic pasting of these
\[
\frac12\binom{2k}k
\]
eight-element Boolean algebras via the atom $A$.
\item[(iv)] Let $D\in Q\setminus\{A\}$. If $A\cap D\neq\emptyset$ then $A\in P_A\setminus P_D$ and hence $P_A\neq P_D$. If $A\cap D=\emptyset$ then
\begin{align*}
\{E\in P_A\mid E\supseteq A\} & =\{A\}\cup\{E\subseteq N\mid A\subseteq E\text{ and }|E|=2k\}\cup\{N\}, \\
\{E\in P_D\mid E\supseteq A\} & =\{A,A\cup D,D',N\}
\end{align*}
and hence
\begin{align*}
|\{E\in P_A\mid E\supseteq A\}| & =1+\binom{2k}k+1=2+\binom{2k}k=2+\frac{2k}k\cdot\frac{2k-1}{k-1}\binom{2k-2}{k-2}> \\
                                & >2+2\cdot1\cdot1=4=|\{E\in P_D\mid E\supseteq A\}|
\end{align*}
whence $P_A\neq P_D$. This shows that the
\[
\binom{2k}k
\]
maximal orthomodular sublattices $\mathbf P_E,E\in Q,$ of the orthomodular poset $\mathbf P_{nk}$ are pairwise distinct.
\end{enumerate}
\end{proof}

From Theorem~\ref{th2} we conclude that the Greechie diagram of the orthomodular lattice $\mathbf P_A$ has the form of a star with
\[
\frac12\binom{2k}k
\]
blocks each of which consists of three atoms.

\section{$\Delta$-blocks}

In what follows we are interested in subsets of an orthoposet which are closed under the orthocomplementation and all the elements of which are in relation $\mathrel\Delta$. We will show that such subsets are of some importance provided they are maximal.

\begin{definition}
A subset $B$ of an orthoposet $\mathbf P=(P,\leq,{}',0,1)$ is called a {\em $\Delta$-block} of $\mathbf P$ if for all $x,y\in B$ we have $x'\in B$ and $x\mathrel\Delta y$ and $B$ is maximal with respect to this property.
\end{definition}

\begin{lemma}
Let $\mathbf P=(P,\leq,{}',0,1)$ be an orthoposet, $B$ a $\Delta$-block of $\mathbf P$ and $a,b\in B$. Then $d(a,b)=\{1\}$.
\end{lemma}

\begin{proof}
According to the definition of a $\Delta$-block we have $a\mathrel\Delta b$ and $a'\mathrel\Delta b$ which implies $d(a,b)=\{1\}$ by (iv) of Lemma~\ref{lem3}.
\end{proof}

The relevance of $\Delta$-blocks is illuminated by the following two results.

\begin{proposition}\label{prop3}
Let $\mathbf P=(P,\leq,{}',0,1)$ be an orthogonal poset and $B$ a $\Delta$-block of $\mathbf P$. Then $B$ is an orthomodular subposet of $\mathbf P$.
\end{proposition}

\begin{proof}
Let $a,b\in B$ with $a\leq b$. Then $a\vee(b\wedge a')$ is defined and $b\mathrel\Delta a$ and hence
\[
U(a\vee(b\wedge a'))=U(a,b\wedge a')=U(L(a),L(b,a'))=U(L(b,a),L(b,a'))=U(b),
\]
i.e.\ $a\vee(b\wedge a')=b$.
\end{proof}

From Proposition~\ref{prop3} we have that every $\Delta$-block of an ortholattice is an orthomodular lattice.

\begin{example}
Consider the orthogonal poset $\mathbf P$ from Example~\ref{ex2}. It is easy to check that it has just four $\Delta$-blocks which are the Boolean lattices $\{0,x,x',1\}$ for each $x\in\{a,b,c,d\}$. One can see that if $x\in P\setminus\{0,1\}$ and $y\in P\setminus\{0,x,x',1\}$ then $x\not\mathrel\Delta y$ or $x'\not\mathrel\Delta y$.
\end{example}

\begin{proposition}
Let $\mathbf P=(P,\leq,{}',0,1)$ be an orthoposet.
\begin{enumerate}[{\rm(i)}]
\item Let $(A,\vee,\wedge,{}',0,1)$ be a Boolean algebra included in $P$. Then there exists a $\Delta$-block $B$ of $\mathbf P$ including $A$.
\item Let $B$ be a $\Delta$-block of $\mathbf P$ and $\mathbf L=(L,\vee,\wedge,{}',0,1)$ be a orthomodular sublattice of $\mathbf P$ contained in $B$. Then $\mathbf L$ is distributive.
\end{enumerate}
\end{proposition}

\begin{proof}
\
\begin{enumerate}[(i)]
\item Obviously, $x'\in A$ and $x\mathrel\Delta y$ for all $x,y\in A$. Hence, according to Zorn's Lemma there exists a maximal subset $B$ of $P$ including $A$ and satisfying $x'\in B$ and $x\mathrel\Delta y$ for all $x,y\in B$. Of course, $B$ is a block of $\mathbf P$.
\item Let $a,b,c\in L$. Since $x\mathrel\Delta y$ for all $x,y\in L$ we have $x=(x\wedge y)\vee(x\wedge y')$ for all $x,y\in L$ and, dually, $x=(x\vee y)\wedge(x\vee y')$ for all $x,y\in L$. Put $d:=(a\wedge c)\vee(b\wedge c)$ and $e:=(a\vee b)\wedge c$. Because of $d\leq e$ we have
\begin{align*}
e & =d\vee(e\wedge d')=d\vee((a\vee b)\wedge c\wedge(a'\vee c')\wedge(b'\vee c'))= \\
  & =d\vee((a\vee b)\wedge(a\vee c)\wedge(a'\vee c)\wedge(b\vee c)\wedge(b'\vee c)\wedge(a'\vee c')\wedge(b'\vee c'))= \\
  & =d\vee((a\vee b)\wedge(a\vee c)\wedge(a'\vee c)\wedge(a'\vee c')\wedge(b\vee c)\wedge(b'\vee c)\wedge(b'\vee c'))= \\
  & =d\vee((a\vee b)\wedge(a\vee c)\wedge a'\wedge(b\vee c)\wedge b')= \\
  & =d\vee((a\vee b)\wedge(a\vee c)\wedge(a'\wedge b')\wedge(b\vee c))=d\vee0=d.
\end{align*}
\end{enumerate}
\end{proof}

\begin{remark}
A similar assertion does not hold when a maximal orthomodular sublattice of an orthogonal poset $\mathbf P$ is considered instead of a Boolean algebra. One can easily check that for the maximal sublattice $\mathbf P_a$ from Example~\ref{ex3} we have
\[
U(e)\neq P=U(\{\emptyset\})=U(\{\emptyset\},\{\emptyset\})=U(L(e,h),L(e,h'))
\]
and hence $e\not\mathrel\Delta h$ which shows that $P_a$ is not contained in any $\Delta$-block of $\mathbf P$.
\end{remark}

Let us recall from \cite{LR} that a {\em poset} $(P,\leq)$ is called {\em modular} if for all $a,b,c\in P$,
\[
a\leq c\text{ implies }L(U(a,b),c)=LU(a,L(b,c))
\]
or, equivalently,
\[
a\leq c\text{ implies }U(a,L(b,c))=UL(U(a,b),c).
\]

The following lemma says that every modular orthogonal poset is orthomodular.

\begin{lemma}\label{lem5}
Let $\mathbf P=(P,\leq,{}',0,1)$ be a modular orthogonal poset. Then $\mathbf P$ is orthomodular.
\end{lemma}

\begin{proof}
If $a,b\in P$ and $a\leq b$ then
\begin{align*}
U(b) & =UL(b)=UL(1,b)=UL(U(a,a'),b)=U(a,L(a',b))=U(a,L(a'\wedge b))= \\
     & =U(a,a'\wedge b)=U(a\vee(a'\wedge b))
\end{align*}
and hence $b=a\vee(a'\wedge b)$.
\end{proof}

However, the converse of Lemma~\ref{lem5} does not hold in general. For example, consider the twenty-element orthomodular poset $\mathbf P$ from Example~\ref{ex1}. Then $\mathbf P$ is not modular because $a\leq i'$, but
\begin{align*}
L(U(a.b'),i') & =L(N,i')=L(i')=\{a,b,d,e\}\neq\{a,d\}=L(h',i')=LU(a,d)= \\
              & =LU(a,L(d))=LU(a,L(b',i')).
\end{align*}
Recall that a {\em poset} $(P,\leq)$ is called {\em distributive} if it satisfies the identity
\[
L(U(x,y),z)\approx LU(L(x,z),L(y,z))
\]
or, equivalently,
\begin{align*}
U(L(x,z),L(y,z)) & \approx UL(U(x,y),z), \\
     U(L(x,y),z) & \approx UL(U(x,z),L(y,z)), \\
L(U(x,z),L(y,z)) & \approx LU(L(x,y),z).
\end{align*}
Of course, every distributive poset is modular. A complemented {\em poset} is called {\em Boolean} if it is distributive.

A Boolean poset need not be orthomodular, see e.g.\ the following one depicted in Fig.~6:
\vspace*{-2mm}
\begin{center}
\setlength{\unitlength}{7mm}
\begin{picture}(8,10)
\put(4,1){\circle*{.3}}
\put(1,3){\circle*{.3}}
\put(3,3){\circle*{.3}}
\put(5,3){\circle*{.3}}
\put(7,3){\circle*{.3}}
\put(1,5){\circle*{.3}}
\put(7,5){\circle*{.3}}
\put(1,7){\circle*{.3}}
\put(3,7){\circle*{.3}}
\put(5,7){\circle*{.3}}
\put(7,7){\circle*{.3}}
\put(4,9){\circle*{.3}}
\put(4,1){\line(-3,2)3}
\put(4,1){\line(-1,2)1}
\put(4,1){\line(1,2)1}
\put(4,1){\line(3,2)3}
\put(4,9){\line(-3,-2)3}
\put(4,9){\line(-1,-2)1}
\put(4,9){\line(1,-2)1}
\put(4,9){\line(3,-2)3}
\put(1,3){\line(0,1)4}
\put(7,3){\line(0,1)4}
\put(1,5){\line(1,1)2}
\put(1,3){\line(1,1)4}
\put(3,3){\line(1,1)4}
\put(5,3){\line(1,1)2}
\put(3,3){\line(-1,1)2}
\put(5,3){\line(-1,1)4}
\put(7,3){\line(-1,1)4}
\put(7,5){\line(-1,1)2}
\put(3.85,.25){$0$}
\put(.3,2.85){$a$}
\put(2.3,2.85){$b$}
\put(5.4,2.85){$c$}
\put(7.4,2.85){$d$}
\put(.3,4.85){$e$}
\put(7.4,4.85){$e'$}
\put(.3,6.85){$d'$}
\put(2.3,6.85){$c'$}
\put(5.4,6.85){$b'$}
\put(7.4,6.85){$a'$}
\put(3.85,9.4){$1$}
\put(3.2,-.75){{\rm Fig.\ 6}}
\end{picture}
\end{center}
\vspace*{4mm}
Here $a\leq d'$, but $a\vee d$ is not defined since $b'$ and $c'$ are different minimal upper bounds of $a$ and $d$. However, Boolean posets that are lattices are orthomodular.

For arbitrary orthoposets we can prove

\begin{lemma}\label{lem6}
Let $\mathbf P=(P,\leq,{}',0,1)$ be an orthoposet, $a,b\in P$ and $(B,\leq,{}',0,1)$ a Boolean subposet of $\mathbf P$. Then {\rm(i)} and {\rm(ii)} hold:
\begin{enumerate}[{\rm(i)}]
\item If $(\{a,b,b'\},\leq)$ is a distributive subposet of $\mathbf P$ then $a\mathrel\Delta b$,
\item $x\mathrel\Delta y$ for all $x,y\in B$,
\item $B$ is contained in some $\Delta$-block of $\mathbf P$.
\end{enumerate}
\end{lemma}

\begin{proof}
\
\begin{enumerate}[(i)]
\item If $(\{a,b,b'\},\leq)$ is distributive then
\begin{align*}
U(a) & =UL(a)=UL(1,a)=UL(U(b.b'),a)=U(L(b,a),L(b',a))= \\
     & =U(L(a,b),L(a,b')),
\end{align*}
i.e.\ $a\mathrel\Delta b$.
\item
If $c,d\in B$ then also $d'\in B$ and hence $c\mathrel\Delta d$ by (i).
\item For all $x,y\in B$ we have $x'\in B$ and by (ii) we have $x\mathrel\Delta y$. Applying Zorn's Lemma we see that $B$ is contained in some $\Delta$-block of $\mathbf P$.
\end{enumerate}
\end{proof}

With respect to Lemma~\ref{lem6} (i) we can ask if a $\Delta$-block of an orthoposet is distributive, i.e.\ if it is a Boolean poset. In what follows we partly solve the problem. At first, we recall the following useful concept introduced by J.~Tkadlec in \cite{T97}.

\begin{definition}\label{def1}
{\rm(}cf.\ {\rm\cite{T97})} Let $\mathbf P=(P,\leq,{}',0,1)$ be an orthomodular poset. Then $\mathbf P$ is called {\em weakly Boolean} if the following condition holds:
\begin{itemize}
\item If $x\wedge y=x\wedge y'=0$ then $x=0$
\end{itemize}
{\rm(}$x,y\in P${\rm)}. Further, $\mathbf P$ is said to have the {\em property of maximality} if for all $x,y\in P$ the set $L(x,y)$ has a maximal element.
\end{definition}

It was shown in \cite{T97}, Theorem 4.2. that every weakly Boolean orthomodular poset having the property of maximality is a Boolean algebra. Using this, we can state the following result.

\begin{theorem}
Let $\mathbf P=(P,\leq,{}',0,1)$ be an orthomodular poset. Then the following hold:
\begin{enumerate}[{\rm(i)}]
\item Every $\Delta$-block of $\mathbf P$ is a weakly Boolean orthomodular subposet of $\mathbf P$,
\item every $\Delta$-block of $\mathbf P$ having the property of maximality is a Boolean subalgebra of $\mathbf P$.
\end{enumerate}
\end{theorem}

\begin{proof}
\
\begin{enumerate}[(i)]
\item Let $B$ be a $\Delta$-block of $\mathbf P$ and $a,b\in B$. Then $a'\in B$. Moreover, if $a\wedge b=a\wedge b'=0$ then $U(a)=U(L(a,b),L(a,b'))=U(0,0)=U(0)$ and hence $a=0$.
\item This follows from (i) and from Theorem~4.2 in \cite{T97}.
\end{enumerate}
\end{proof}

Hence, for orthomodular posets we know that every $\Delta$-block of is a weakly Boolean poset, and, in a particular case, it is a Boolean algebra. Fig.~7 below shows an example of an orthoposet $(P,\leq,{}',0,1)$ which satisfies the conditions of Definition~\ref{def1}, but which is not orthomodular, it is not even an orthogonal poset. On the other hand, it is distributive and hence Boolean thus it satisfies $x\Delta y$ for all $x,y\in P$ according to Lemma~\ref{lem6}. The example also shows that a $\Delta$-block of an orthoposet need not be a sublattice.

\begin{example}
The Boolean poset $\mathbf P=(P,\leq,{}',0,1)$ depicted in Fig.~7
\vspace*{-2mm}
\begin{center}
\setlength{\unitlength}{7mm}
\begin{picture}(8,8)
\put(4,1){\circle*{.3}}
\put(1,3){\circle*{.3}}
\put(3,3){\circle*{.3}}
\put(5,3){\circle*{.3}}
\put(7,3){\circle*{.3}}
\put(1,5){\circle*{.3}}
\put(3,5){\circle*{.3}}
\put(5,5){\circle*{.3}}
\put(7,5){\circle*{.3}}
\put(4,7){\circle*{.3}}
\put(4,1){\line(-3,2)3}
\put(4,1){\line(-1,2)1}
\put(4,1){\line(1,2)1}
\put(4,1){\line(3,2)3}
\put(4,7){\line(-3,-2)3}
\put(4,7){\line(-1,-2)1}
\put(4,7){\line(1,-2)1}
\put(4,7){\line(3,-2)3}
\put(1,3){\line(0,1)2}
\put(1,3){\line(1,1)2}
\put(1,3){\line(2,1)4}
\put(3,3){\line(-1,1)2}
\put(3,3){\line(0,1)2}
\put(3,3){\line(2,1)4}
\put(5,3){\line(-2,1)4}
\put(5,3){\line(0,1)2}
\put(5,3){\line(1,1)2}
\put(7,3){\line(-2,1)4}
\put(7,3){\line(-1,1)2}
\put(7,3){\line(0,1)2}
\put(3.85,.25){$0$}
\put(.3,2.85){$a$}
\put(2.3,2.85){$b$}
\put(5.4,2.85){$c$}
\put(7.4,2.85){$d$}
\put(.3,4.85){$d'$}
\put(2.3,4.85){$c'$}
\put(5.4,4.85){$b'$}
\put(7.4,4.85){$a'$}
\put(3.85,7.4){$1$}
\put(3.2,-.75){{\rm Fig.\ 7}}
\end{picture}
\end{center}
\vspace*{4mm}
is a $\Delta$-block of itself and it is not a lattice.
\end{example}

\begin{lemma}
Let $\mathbf P=(P,\leq,{}',0,1)$ be an orthoposet satisfying $x\mathrel\Delta y$ for all $x,y\in P$. Then $\mathbf P$ does not contain a subposet isomorphic to $\mathbf O_6$ {\rm(}see Fig.~2{\rm)}.
\end{lemma}

\begin{proof}
If $\mathbf P$ would contain a subposet isomorphic to $\mathbf O_6$ then we would have $b\mathrel\Delta a$, i.e.
\[
b\in U(b)=U(L(b,a),L(b,a'))=U(L(a),L(0))=U(a,0)=U(a),
\]
a contradiction.
\end{proof}

Finally, we add two more result concerning intervals of orthomodular posets. The following result was proved in \cite{PP} (Propositions~1.3.6 and 1.3.12).

\begin{theorem}\label{th3}
Let $(P,\leq,{}',0,1)$ be an orthomodular poset and $a,b\in P$ with $a\leq b$. Then $(x'\vee a)\wedge b=(x'\wedge b)\vee a$ for all $x\in[a,b]$. Put $x^+:=(x'\vee a)\wedge b$ for all $x\in[a,b]$. Then $([a,b],\leq,{}^+,a,b)$ is an orthomodular poset.
\end{theorem}

\begin{example}
Consider the orthomodular poset $\mathbf P$ of Example~\ref{ex1}. Then the Hasse diagram of $([\emptyset,i'],\subseteq,{}^+,\emptyset,i')$ is depicted in Fig.~8:
\vspace*{-2mm}
\begin{center}
\setlength{\unitlength}{7mm}
\begin{picture}(8,6)
\put(4,1){\circle*{.3}}
\put(1,3){\circle*{.3}}
\put(3,3){\circle*{.3}}
\put(5,3){\circle*{.3}}
\put(7,3){\circle*{.3}}
\put(4,5){\circle*{.3}}
\put(4,1){\line(-3,2)3}
\put(4,1){\line(-1,2)1}
\put(4,1){\line(1,2)1}
\put(4,1){\line(3,2)3}
\put(4,5){\line(-3,-2)3}
\put(4,5){\line(-1,-2)1}
\put(4,5){\line(1,-2)1}
\put(4,5){\line(3,-2)3}
\put(3.85,.25){$\emptyset$}
\put(.3,2.85){$a$}
\put(2.4,2.85){$b$}
\put(5.4,2.85){$d$}
\put(7.4,2.85){$e$}
\put(3.85,5.4){$i'$}
\put(3.2,-.75){{\rm Fig.\ 8}}
\end{picture}
\end{center}
\vspace*{4mm}
We have
\[
\begin{array}{r|rrrrrr}
  x & \emptyset & a & b & d & e &        i' \\
\hline
x^+ &        i' & e & d & b & a & \emptyset
\end{array}
\]
\end{example}

Now we show when an orthomodular poset $\mathbf P$ can be embedded into a direct product of intervals of $\mathbf P$. For this, we define the following concept.

\begin{definition}
We call an element $c$ of an orthoposet $(P,\leq,{}',0,1)$ {\em central} if $x\mathrel\Delta c$ for all $x\in P$ and if, moreover, for any $x\in P$ the infima $x\wedge c$ and $x\wedge c'$ exist.
\end{definition}

\begin{theorem}\label{th4}
Let $\mathbf P=(P,\leq,{}',0,1)$ be an orthomodular poset and $c\in P$.
\begin{enumerate}[{\rm(i)}]
\item If $c$ is central then $\mathbf P$ can be embedded into $[0,c]\times[0,c']$.
\item If $c$ is central and the following condition holds:
\begin{enumerate}
\item[{\rm(1)}] If $x,y\in P$, $x\leq c$ and $y\leq c'$ then $x\vee y$ is defined, $(x\vee y)\wedge c=x$ and $(x\vee y)\wedge c'=y$.
\end{enumerate}
then $\mathbf P\cong[0,c]\times[0,c']$.
\end{enumerate}
\end{theorem}

\begin{proof}
According to Theorem~\ref{th3}, the intervals $[0,c]$ and $[0,c']$ of $\mathbf P$ and hence also their direct product can be considered as orthomodular posets in a canonical way.
\begin{enumerate}[(i)]
\item Assume $c$ to be central. Since
\[
U(x)=U(L(x,c),L(x,c'))=U(L(x\wedge c),L(x\wedge c'))=U(x\wedge c,x\wedge c')
\]
for all $x\in P$ we have
\[
x=(x\wedge c)\vee(x\wedge c')
\]
for all $x\in P$ and hence, using De Morgan's laws,
\[
x=(x\vee c)\wedge(x\vee c')
\]
for all $x\in P$. Now let $a,b\in P$. Further, let $f$ denote the mapping from $P$ to $[0,c]\times[0,c']$ defined by $f(x):=(x\wedge c,x\wedge c')$ for all $x\in P$. If $a\leq b$ then
\[
f(a)=(a\wedge c,a\wedge c')\leq(b\wedge c,b\wedge c')=f(b).
\]
If, conversely, $f(a)\leq f(b)$ then $a\wedge c\leq b\wedge c$ and $a\wedge c'\leq b\wedge c'$ and hence
\[
a=(a\wedge c)\vee(a\wedge c')\leq(b\wedge c)\vee(b\wedge c')=b.
\]
This shows that $a\leq b$ if and only if $f(a)\leq f(b)$. From this we conclude that $f$ is injective. Finally, using $a'=(a'\vee c)\wedge(a'\vee c')$ we have
\begin{align*}
f(a') & =(a'\wedge c,a'\wedge c')=(((a'\vee c)\wedge(a'\vee c'))\wedge c,((a'\vee c)\wedge(a'\vee c'))\wedge c')= \\
      & =((a'\vee c)\wedge(a'\vee c')\wedge c,(a'\vee c)\wedge(a'\vee c')\wedge c')=((a'\vee c')\wedge c,(a'\vee c)\wedge c')= \\
      & =((a\wedge c)'\wedge c,(a\wedge c')'\wedge c')=(f(a))', \\
 f(0) & =(0,0), \\
 f(1) & =(c,c').
\end{align*}
\item Assume $c$ to be central and (1) to hold. Then all what was proved in (i) holds. Let $f$ be defined as in (i) and let $g$ denote the mapping from $[0,c]\times[0,c']$ to $P$ defined by $g(x,y):=x\vee y$ for all $(x,y)\in[0,c]\times[0,c']$. Then
\begin{align*}
  g(f(x)) & =(x\wedge c)\vee(x\wedge c')=x\text{ for all }x\in P, \\
f(g(x,y)) & =((x\vee y)\wedge c,(x\vee y)\wedge c')=(x,y)\text{ for all }(x,y)\in[0,c]\times[0,c'].
\end{align*}
This shows that $f$ and $g$ are mutually inverse bijections between $P$ and $[0,c]\times[0,c']$. Since according to the proof of (i), $f$ is a homomorphism from $\mathbf P$ to $[0,c]\times[0,c']$ satisfying
\[
x\leq y\text{ if and only if }f(x)\leq f(y)
\]
for all $x,y\in P$, $f$ and $g$ are mutually inverse isomorphisms between $\mathbf P$ and $[0,c]\times[0,c']$.
\end{enumerate}
\end{proof}

\begin{remark}
Condition {\rm(1)} of Theorem~\ref{th4} seems to be rather restrictive but this is not the case. If $c\in P$ then $(c,0)$ and $(0,c')$ {\rm(}$=(c'\wedge c,0'\wedge c')=(c,0)'${\rm)} are central elements of $[0,c]\times[0,c']$. Moreover, if
\begin{align*}
(x,0),(0,y) & \in[0,c]\times[0,c'], \\
      (x,0) & \leq(c,0), \\
      (0,y) & \leq(0,c')
\end{align*}
then
\begin{align*}
              (x,0)\vee(0,y) & =(x,y), \\
 ((x,0)\vee(0,y))\wedge(c,0) & =(x,0), \\
((x,0)\vee(0,y))\wedge(0,c') & =(0,y).
\end{align*}
This can be seen as follows: If $(x,y)\in[0,c]\times[0,c']$ then
\begin{align*}
           U(x,y) & =U((x,0),(0,y))=U(L(x,0),L(0,y))=U(L(x)\times L(0),L(0)\times L(y))= \\
                  & =U(L(x,c)\times L(y,0),L(x,0)\times L(y,c'))= \\
									& =U(L((x,y),(c,0)),L((x,y),(0,c'))), \\
            (x,y) & \mathrel\Delta(c,0), \\
 (x,y)\wedge(c,0) & =(x,0), \\
(x,y)\wedge(0,c') & =(0,y)
\end{align*}
showing that $(c,0)$ is a central element of $[0,c]\times[0,c']$. Dually, we have that also $(0,c')$ is a central element of $[0,c]\times[0,c']$. Finally, if
\begin{align*}
(x,0),(0,y) & \in[0,c]\times[0,c'], \\
      (x,0) & \leq(c,0), \\
      (0,y) & \leq(0,c')
\end{align*}
then
\begin{align*}
              (x,0)\vee(0,y) & =(x,y), \\
 ((x,0)\vee(0,y))\wedge(c,0) & =(x,y)\wedge(c,0)=(x,0), \\
((x,0)\vee(0,y))\wedge(0,c') & =(x,y)\wedge(0,c')=(0,y).
\end{align*}
\end{remark}

Authors' addresses:

Ivan Chajda \\
Palack\'y University Olomouc \\
Faculty of Science \\
Department of Algebra and Geometry \\
17.\ listopadu 12 \\
771 46 Olomouc \\
Czech Republic \\
ivan.chajda@upol.cz

Helmut L\"anger \\
TU Wien \\
Faculty of Mathematics and Geoinformation \\
Institute of Discrete Mathematics and Geometry \\
Wiedner Hauptstra\ss e 8-10 \\
1040 Vienna \\
Austria, and \\
Palack\'y University Olomouc \\
Faculty of Science \\
Department of Algebra and Geometry \\
17.\ listopadu 12 \\
771 46 Olomouc \\
Czech Republic \\
helmut.laenger@tuwien.ac.at
\end{document}